\documentclass[12pt]{amsart}

\usepackage{a4wide}
\usepackage{amssymb,amsmath,latexsym,color,graphicx,stmaryrd}
\usepackage{hyperref}

\makeatletter
\@addtoreset{equation}{section}
\makeatother

\newtheoremstyle{tr}
{3pt}
{3pt}
{\sl}
{}
{\bf}
{.---}
{.5em}
{}

\theoremstyle{tr}

\newtheorem{dref}{Definition}[section] 
\newtheorem{theo}[dref]{Theorem} \newtheorem{prop}[dref]{Proposition}

\renewcommand\Re{\mathop{\rm Re}\nolimits}
\renewcommand\Im{\mathop{\rm Im}\nolimits}

\def\CC{{\mathcal C}}
\def\CD{{\mathcal D}}
\def\CE{{\mathcal E}}
\def\CL{{\mathcal L}}
\def\CO{{\mathcal O}}
\def\CP{{\mathcal P}}
\def\CT{{\mathcal T}}

\def\C{{\mathbb C}}
 
\def\R{{\mathbb R}}

\title{Non-real eigenvalues for ${\CP\CT}$-symmetric double wells}

\author[A.B.]{Amina Benbernou}
\address{Amina Benbernou, Facult\'{e} des sciences exactes et informatique, Universit\'{e} de Mostaganem, 27000-Mostaganem, Alg\'erie}
\email{abenbernou@yahoo.fr}

\author[N.B.]{Naima Boussekkine}
\address{Naima Boussekkine, Facult\'{e} des sciences exactes et informatique, Universit\'{e} de Mostaganem, 27000-Mostaganem, Alg\'erie}
\email{nboussekkine@yahoo.fr}

\author[N.M.]{Nawal Mecherout}
\address{Nawal Mecherout, Facult\'{e} des sciences exactes et informatique, Universit\'{e} de Mostaganem, 27000-Mostaganem, Alg\'erie}
\email{mecheroutnawel@yahoo.fr}

\author[T.R.]{Thierry Ramond}
\address{Thierry Ramond, Laboratoire de Math\'ematiques d'Orsay (UMR CNRS 8628), Universit\'e Paris Sud 11, 91400 Orsay, France}
\email{thierry.ramond@math.u-psud.fr}

\author[J.S.]{Johannes Sj¨\"ostrand}
\address{Johannes Sj¨\"ostrand,
Institut de Math\'ematiques de Bourgogne (UMR 5584 du CNRS), Universit\'e de Bourgogne, 9 avenue Alain Savary - BP 47870, 21078 Dijon cedex, France}
\email{johannes.sjostrand@u-bourgogne.fr}
\thanks{T.R. and J.S. are  partially supported by the ANR project NOSEVOL ANR 2011 BS 010119 01.}

\begin{document}

\begin{abstract} We study small, $\CP\CT$-symmetric perturbations of
  self-adjoint double-well Schr\"o\-dinger operators in dimension
  $n\ge 1$.  We prove that the eigenvalues stay real for a very small
  perturbation, then bifurcate to the complex plane as the
  perturbation gets stronger.
\end{abstract}

\maketitle

\centerline{In memory of Louis Boutet de Monvel}


\tableofcontents


\section{Introduction}
\label{int}

We study spectral properties of small, $\CP\CT$-symmetric perturbations of self-adjoint double-well Schr\"odinger operators 
\begin{equation}\label{int.1}
P_\varepsilon =-h^2\Delta +V_\varepsilon,
\end{equation}
on $M$, a smooth compact Riemannian manifold of dimension $n$, or
$\R^n$, where the potential is of the form
\begin{equation}\label{int.2}
V_\varepsilon (x)=V_0(x)+i\varepsilon W(x).
\end{equation}
Here $\varepsilon \in {\R}$, $|\varepsilon | \ll 1$,
 $V_0,W\in \CC^\infty (M;{\R})$ and $W$
is bounded. $\Delta $ denotes the Laplace-Beltrami operator on $M$.

The one dimensional-case has been considered in \cite{MeBoRaSj} under
an additional assumption of analyticity, and we concentrate here on
the general $n$-dimensional case, $n>1$. Our new result is more
general, but requires a stronger condition on the size of the
perturbation parameter.

To be more precise, $P_0$ denotes the Friedrichs extension of the differential operator $-h^2\Delta+V_{0}$ from
$\CC_0^\infty (M)$. In the case $M=\R^n$, it is well-known that 
\begin{equation}\label{int.3}
\inf \sigma _\mathrm{ess}(P_0)\ge \liminf_{x\to \infty }V_0(x)=:\alpha,
\end{equation}
or in other words, that the spectrum of $P_0$ in $]-\infty ,\alpha [$ is
purely discrete. This assertion is also true when $M$ is a compact manifold, with $\alpha=+\infty$ in that case.

Since $W$ is bounded, we can define $P_\varepsilon = P_{0} +i\varepsilon W$ as a closed operator with the same
domain as $P_0$, and it is proved in Proposition \ref{app.1} below that the spectrum of $P_{\varepsilon}$ is discrete in the half-plane $\{z\in \C,\ \Re z<\alpha\}$. To fix the ideas, we will assume when $M={\R}^n$ that 
\begin{equation}\label{int.4}
\alpha>0.
\end{equation}
Thus there exists an $h$-independent neighborhood $\CE$ of $E_{0}=0$
in $\C$ such that $P_{0}$ and $P_{\varepsilon}$ have only discrete
spectrum in $\CE$.

We shall also assume that we have an isometry $\iota :M\to M$, different from the
identity, such that
\begin{equation}\label{int.5}
\iota ^2=\mathrm{id},
\end{equation}
and
\begin{equation}\label{int.6}
  V_0\circ \iota =V_0.\end{equation}
We suppose further that $V_{0}$ has a double-well structure at energy $E_{0}=0$, and that the two wells are exchanged by $\iota$. More precisely,  we assume that
\begin{equation}\label{int.7}
V_0^{-1}(]-\infty ,0])=U_{-1}\cup U_1,\quad U_{-1}\cap U_1=\emptyset ,
\end{equation}
where $U_{\pm 1}\subset M$ are non-empty, closed and hence compact in
view of the assumption (\ref{int.4}), and that
\begin{equation}\label{int.8}
\iota (U_{-1})=U_1.
\end{equation}

In Section \ref{la} we review some basic facts
about the Lithner-Agmon metric $V_0(x)_+dx^2$ (cf.~ (\ref{la.4})) and
the corresponding distance $d(x,y)$, which may be degenerate in the
sense that $d(x,y)$ may be zero when $x\ne y$, but which is symmetric
and satisfies the triangle inequality (cf.~(\ref{la.5}), (\ref{la.6}))
and is a locally Lipschitz function (cf.~(\ref{la.7})--(\ref{la.9})). 

Let $\mathrm{diam}_d(U_j)$ denote the diameter of $U_j$ with respect
to $d$. Then the two diameters are equal and we assume that
\begin{equation}\label{int.8.5}
\mathrm{diam}_d(U_j)=0,\ j=\pm 1.
\end{equation}
 
\par To describe the spectrum of $P_{0}$, it is convenient to
introduce two self-adjoint reference operators. Let $\chi _{\pm 1}\in
\CC_0^\infty (M;[0,1])$ have the following properties:
\begin{equation}\label{int.9}
\chi _j=1\hbox{ near }U_j,
\end{equation}
\begin{equation}\label{int.10}
\mathrm{supp\,}\chi _j\subset B(U_j,\delta )=:U_j^\delta 
\end{equation}
where $\delta >0$ is small. Here 
$$
B(U_j,\delta )=\{x\in M;\, d(U_j,x)<\delta  \}.
$$
Put
\begin{equation}\label{int.11}
\widetilde{P}_j=\widetilde{P}_{0,j}=P_0+\lambda \chi _{-j},\ j=\pm 1.
\end{equation}
Here $\lambda >0$ is a constant that we choose large enough so that 
$$
\{x\in M;\, V_0(x)+\lambda \chi _{-j}(x)\le 0 \}=U_j,
$$
and hence the effect of adding $\lambda \chi _{-j}$ to $V_0$ is to
fill the well $U_{-j}$. If we define
\begin{equation}\label{int.12}
{\CP}u=u\circ \iota ,\ u\in L^2(M),
\end{equation}
then ${\CP}$ is unitary on $L^2(M)$ with ${\CP}\ne 1={\CP}^2$
and we have
\begin{equation}\label{int.13}
{\CP}\circ P_0=P_0\circ {\CP},
\end{equation}
\begin{equation}\label{int.14}
{\CP}\circ \widetilde{P}_j=\widetilde{P }_{-j},\ j=\pm 1.
\end{equation}
The last relation implies in particular that $\widetilde{P}_{-1}$ and
$\widetilde{P}_1$ have the same spectrum. 

\par Assume that 
\begin{equation}\label{int.15}
\widetilde{\mu }(h)=o(h)
\end{equation}
is a simple eigenvalue of $\widetilde{P}_1$ (and hence of
$\widetilde{P}_{-1}$), and  that  
\begin{equation}
\exists C_{0},N_{0}>0,\ \sigma(\widetilde{P}_{\pm 1})\ \cap\
]\widetilde{\mu }(h)-h^{N_0}/C_0,\widetilde{\mu}(h)+h^{N_0}/C_0[=\{
\widetilde{\mu }(h) \}.
\end{equation}

\par As we shall review in Section \ref{pmr}, if $\delta >0$ is small
enough, then for $h>0$ small enough, $P_0$ has exactly two eigenvalues
in the interval
$$
]\widetilde{\mu } (h)-h^{N_0}/(2C_0), \widetilde{\mu }
(h)+h^{N_0}/(2C_0) [, 
$$
namely the eigenvalues $\mu (h)\pm |t(h)|$ of the matrix in (\ref{pmr.8}),
$$
\begin{pmatrix} \mu (h) &t(h)\\ \overline{t(h)} &\mu (h)\end{pmatrix} ,
$$
where $\mu (h)\in \mathbb{R}$, $t(h)\in \mathbb{C}$ satisfy for all
$\delta >0$,
$$
\mu (h)=\widetilde{\mu }(h)+\mathcal{ O}_\delta (e^{(\epsilon (\delta
  )-2S_0)/h }),\ \ \epsilon (\delta )\to 0,\ \delta \to 0,
$$
$$
\forall \alpha >0,\ t(h)=\mathcal{ O}_\alpha (e^{(\alpha -S_0)/h}).
$$
Here, the constant $S_{0}$ is the Lithner-Agmon distance between the two wells $U_{\pm 1}$:
\begin{equation}\label{int.15.5}
S_{0}=d(U_{1},U_{-1}).
\end{equation}
As a matter of fact, quite often we also have a lower bound on
$|t(h)|$:
$$
\forall \alpha >0,\ |t(h)|^{-1}=\mathcal{ O}_\alpha (e^{(\alpha +S_0)/h}).
$$
There are nowadays a lot of precise results available on the tunneling coefficient $t(h)$. One may refer for example to \cite{Si}, \cite {HeSj} or to the review paper \cite{Sj} and the references therein.

Concserning the perturbation $W$, we assume also that
\begin{equation}\label{int.16}
W\circ \iota =-W.
\end{equation}
Then ${\CP}P_\varepsilon =P_{-\varepsilon }{\CP}$, where we also
remark that $P_{-\varepsilon} =P_\varepsilon ^*$. 
Now if we denote by ${\CT}$ the anti-linear operator defined by
\begin{equation}
{\CT}u(x)=\overline{u(\bar x)},
\end{equation}
we see that ${\CT}P_\varepsilon
=P_{-\varepsilon }{\CT}$, so that  $P_\varepsilon $ is ${
\CP\CT}$-symmetric:
\begin{equation}\label{int.17}
{\CP\CT}P_\varepsilon =P_\varepsilon {\CP\CT}.
\end{equation}

The main result of this paper is the following 

\begin{theo}\label{int.18} Under the above assumptions, the operator $P_{\varepsilon}$ has exactly two eigenvalues  (counted with their algebraic multiplicity) in 
$D(\widetilde{\mu},h^{N_0}/C)$ for $C\gg 0$ and for $\varepsilon $ real such that
$|\varepsilon |\ll h^{N_0}$. These eigenvalues are equal to the
eigenvalues of the matrix
$$
M_\varepsilon =\begin{pmatrix}a(\varepsilon ) &b(\varepsilon )\\
\overline{b}(\varepsilon ) &\overline{a}(\varepsilon ),
\end{pmatrix}
$$
and hence of the form
$$
\lambda _\pm = \Re a\pm \sqrt{|b|^2-(\Im a)^2}.
$$
Here $a(\varepsilon )=a(\varepsilon ;h)$, $b(\varepsilon
)=b(\varepsilon ;h)$ satisfy,
$$
a(0;h)=\mu (h),\ \ b(0;h)=t(h),
$$
$$
\partial _\varepsilon a =i\int W(x)|e_1^0(x)|^2dx+\mathcal{ O}(\varepsilon
h^{-N_0})+\mathcal{ O}_\delta (e^{(\epsilon (\delta )-2S_0)/h}),
$$
$$
\partial _{\varepsilon }b=\mathcal{ O}_\delta e^{(\epsilon (\delta )-S_0)/h},
$$
for all $\delta >0$, where $\epsilon (\delta )\to 0$, $\delta \to
0$. Further, $e_1^0$ is the normalized eigenfunction with
$(\widetilde{P}_1-\widetilde{\mu }(h))e_1^0=0$.

\par If $W>0$ on $U_1$, then 
\begin{equation}\label{int.19}
\int W(x)|e_1^0(x)|^2 dx\asymp 1,
\end{equation}
and if we assume that (\ref{int.19}) holds, then there exists $\varepsilon _+\ge 0$ with
the asymptotics,
\begin{equation*}
  \varepsilon _+=(1+\widetilde{\CO}_\delta (e^{(\epsilon (\delta
    )-S_0)/h}))\frac{|t(h)|}{\int W(x)|e_1^0(x)|^2dx},\ \ \epsilon
  (\delta )\to 0,\ \delta \to 0,
\end{equation*}
such that 
\begin{itemize}
\item The two eigenvalues are real and distinct for $|\varepsilon
  |<\varepsilon _+$.
\item They are double and real when $|\varepsilon |=\varepsilon _+$.
\item They are non-real and complex conjugate, when $\varepsilon
  _+<|\varepsilon |\ll h^{N_0}$.
\end{itemize}
\end{theo}

\section{Lithner-Agmon estimates for non-self-adjoint Schr\"odinger operators}
\label{la}

We will need a few extensions of the
tunneling theory in the spirit of B.~Helffer and J.~Sj\"ostrand \cite{HeSj} to the
case of non-self-adjoint Schr\"odinger operators. We will follow the
presentation in Chapter 6 in \cite{DiSj99}. In the following $M$ will
denote either ${\R}^n$ or a compact Riemannian manifold. We start
by reviewing exponentially weighted Lithner-Agmon estimates. The
following is an immediate extension of Proposition 6.1 in
\cite{DiSj99}.

\begin{prop}\label{la.1}
  Let $\Omega \Subset M$ be open with smooth boundary and put
  $P=-h^2\Delta +V(x)$, for some fixed $V\in \CC(\overline{\Omega };{
    \C})$. Let $\Phi \in \CC^2(\overline{\Omega };{\R})$ Then for
  every $u\in \CC^2(\overline{\Omega })$ with ${{u}_\vert}_{\partial
    \Omega }=0$, we have
\begin{equation}\label{la.2}
\begin{split}
&h^2\int_\Omega | \nabla (e^{\Phi /h}u)|^2 dx+\int_{\Omega }(\Re
V(x)-|\nabla \Phi (x)|^2)e^{2\Phi (x)/h}|u(x)|^2 dx\\
&=\Re \int_\Omega
e^{2\Phi (x)/h}Pu(x)\overline{u}(x)dx.
\end{split}
\end{equation}
Here $|\cdot |$ denotes the standard norm on scalars or vectors. In
the Riemannian case the norm of the gradient is the natural one for
cotangent vectors. $\Delta $ denotes the Laplace-Beltrami operator and
$dx$ is the natural volume element.
\end{prop} 
Proposition 6.2 in \cite{DiSj99} extends to:

\begin{prop}\sl\label{la.3}
Under the assumptions of Proposition \ref{la.1}, let $0\le F_{\pm}\in L^\infty (\Omega )$
and $\Phi \in \CC^2(\overline{\Omega };{\R})$
satisfy
$$
\Re V -(\nabla \Phi (x))^2=F_+(x)^2-F_-(x)^2 \hbox{ almost everywhere.}
$$
Then \[
h^2\| \nabla (e^{\Phi /h}u)\|^2+\frac{1}{2}\| F_+e^{\Phi /h}u\|^2\le
\|\frac{1}{F_++F_-}e^{\Phi /h}Pu\|^2+\frac{3}{2}\|F_-e^{\Phi /h}u\|^2 .
\]
\end{prop}

\par The propositions \ref{la.1}, \ref{la.3} allow us to make an
immediate extension of the discussion of the Lithner-Agmon (that we
abbreviate with LA) metric
(originally introduced in \cite{Li} and \cite{Ag}) and Proposition 6.4
in \cite{DiSj99}. We just have to replace the real potential there by
the real part $V_0$ of the potential $V_\varepsilon $ and recall that we
work near the real energy level $0$. We repeat the discussion for
completeness.

The LA metric is defined to be 
\begin{equation}\label{la.4}
V_0(x)_+dx^2.
\end{equation}
For a $\CC^1$ curve $\gamma $ we let $|\gamma |$ denote its
length in the LA-metric. If $x,y\in M$ we define the LA distance
$d(x,y)$ between $x$ and $y$ to be the infimum of the lengths $|\gamma
|$ for all $\CC^1$ curves from  $y$ to $x$. This distance may be
degenerate in the sense that we may have $d(x,y)=0$ for distinct
points $x$ and $y$. Nevertheless:
\begin{equation}\label{la.5}
d(x,y)=d(y,x),\ d(x,z)\le d(x,y)+d(y,z),
\end{equation}
\begin{equation}\label{la.6}
|d(x,z)-d(x,y)|\le d(y,z).
\end{equation}
Further, $y\mapsto d(x,y)$ is a locally Lipschitz function and 
\begin{equation}\label{la.7}
|d(x,z)-d(x,y)|\le (V_0(y)_++o(1))^{\frac{1}{2}}|z-y|_y,
\end{equation}
when $z\to y$, where $|.|_y$ is the Riemannian norm on $T_yM$ and we
identify $\mathrm{neigh\,}(0,T_yM)$ with $\mathrm{neigh\,}(y,M)$ by
means of the exponential map. It follows that for all $x,y\in M$,
\begin{equation}\label{la.8}
|\nabla _yd(x,y)|\le V_0(y)_+^{\frac{1}{2}},
\end{equation} 
\begin{equation}\label{la.9}
|\nabla _xd(x,y)|\le V_0(x)_+^{\frac{1}{2}}.
\end{equation} 
If $U\subset M$, we put
$d(x,U)=\inf _{y\in U}d(x,y)$. 
Then $|d(x,U)-d(y,U)|\le d(x,y)$, so $|\nabla _xd(x,U)|\le
V_0(x)_+^{\frac{1}{2}}$ a.e. on $M$.

Proposition 6.4 in \cite{DiSj99} remains valid, but we prefer to give
the following variant whose proof is basically the same:
\begin{prop}\label{la.10}
  Let $\CE\subset {\R}$, $K\subset M$ be compact sets,
  $0<h_0<<1$ and assume that
$$
(P_\varepsilon -z)u=v,\quad z=z(h)\to 0 \mbox{ as } h\to 0
$$
where $\varepsilon =\varepsilon (h)\in \CE$, $u=u(h)\in \CD(P_0)$,
$v=v(h)\in L^2$, $\mathrm{supp\,}v\subset K$. Then for every fixed
$\delta >0$ there exists a constant $C_\delta $ (independent of $u$, $v$) such that
$$
\|\nabla (e^{(1-\delta )\widetilde{\Phi }/h}u)\| +\|e^{(1-\delta
  )\widetilde{\Phi }/h}(1+V_{0\,+}^{\frac{1}{2}})u\|\le C_\delta
e^{\delta /h}\|u\|_{H^1(K_\delta )},
$$
where 
$$
K_\delta =\{x\in M;\, d_M(x,K)<\delta  \},\ \widetilde{\Phi }(x)=d(U,x).
$$
Here $d_M$ denotes the Riemannian distance.
\end{prop}

\par We end this section by recalling some terminology from
\cite{DiSj99} (earlier used in the works of Helffer and Sj\"ostrand,
cf. \cite{HeSj}). Let $A=A_h$ be a family of operators $L^2(M)\to H^1(M)$ depending on $h\in
]0,h_0[$ where $h_0>0$ is small. Let $f\in C^0(M\times M;{\R})$. We
say that the kernel $A(x,y)$ of $A$ (using the same notation for an
operator and its distribution kernel) is $\widehat{{\mathcal O}}(e^{-f(x,y)/h})$ if for all $x_0,y_0\in M$ and $\delta >0$,
there exist neighborhoods $V,U\subset M$ of $x_0$ and $y_0$ and a
constant $C>0$, such that 
$$
\|Au\|_{H^1(V)}\le Ce^{-(f(x_0,y_0)-\delta )/h}\|u\|_{L^2(U)}
$$
for all $u\in L^2(M)$ with support in $U$. We have the analogous
definitions for operators $L^2(M)\to L^2(M)$ and the choice of arrival
space will be clear from the context. If not, we write $\widehat{{\mathcal O}}_{L^2\to H^1}(e^{-f/h})$ and $\widehat{{\mathcal O}}_{L^2\to L^2}(e^{-f/h})$, to specify. 

We make two observations in the case when $M$ is compact
\begin{itemize}
\item[1)] If $A(x,y)=\widehat{\CO}_{L^2\to X}(e^{-f/h})$,
  $B(x,y)=\widehat{\CO}_{L^2\to L^2}(e^{-g/h})$, where $X$ is
  equal to $L^2$ or $H^1$, then $A\circ B(x,y)=\widehat{{\mathcal O}}_{L^2\to X}(e^{-k/h})$, where $k(x,y)=\min_{x\in
    M}(f(x,z)+g(z,y))$.
\item[2)] There is an obviously analogous notion $u=\widehat{{\mathcal O}}_X (e^{\phi (x)/h})$ when $\phi \in \CC(M;{\R})$, $u\in X$,
  $X=L^2$ or $X=H^1$. Let $A(x,y)=\widehat{\CO}_{L^2\to
    X}(e^{-f(x,y)/h})$, $u=\widehat{\CO}_{L^2}(e^{\phi /h})$
  where $\phi \in \CC(M;{\R})$.  Then, $Au=\widehat{{\mathcal O}}_X(e^{\psi/h} )$, where $\psi (x)=\sup_{y\in M}(-k(x,y)+\phi
  (y))$.
\end{itemize}
When $M={\R}^n$, one can adapt these notions provided that we have
some uniform exponential decay near infinity. Below, we will always be
in such situations, so we shall proceed as in the compact case.

\section{Proof of the main result}\label{pmr}
\setcounter{equation}{0}

 Let $\widetilde{e}_j=\widetilde{e}_j(h)$ be normalized
eigenfunctions of $\widetilde{P}_j$ corresponding to the eigenvalue
$\mu (h)$:
\begin{equation}\label{pmr.1}
(\widetilde{P}_j-\widetilde{\mu })\widetilde{e}_j=0.
\end{equation}
We choose $\widetilde{e}_j$ so that
\begin{equation}\label{pmr.2}
{\CP}\widetilde{e}_j=\widetilde{e}_{-j}.
\end{equation}

\par We know that
\begin{equation}\label{pmr.3}
\widetilde{e}_j=\widehat{\CO}_{H^1}(e^{-d(U_j,x)/h}),
\end{equation}
and we have nice uniform exponential decay estimates near infinity
when $M={\R}^n$ (cf. Proposition \ref{la.3}). In particular,
\begin{equation}\label{pmr.4}
(\widetilde{e}_1|\widetilde{e}_{-1})=\widehat{\CO}(e^{-S_0/h}),
\end{equation}
where we extended the notion $\widehat{\CO}$ to scalar quantities
in the natural way. 

\par We know that for $h$ small enough, the spectrum of $P_0$ in
\begin{equation}\label{pmr.5}
]\widetilde{\mu }-\frac{h^{N_0}}{2C_0},\widetilde{\mu }+\frac{h^{N_0}}{2C_0}[\end{equation} consists of two simple or one double double
eigenvalue. Let $\CE_0(h)\subset L^2(M)$ be the corresponding
2-dimensional spectral subspace and let $\Pi _0(h):L^2(M)\to L^2(M)$
be the associated spectral projection. Since $P_0$ is self-adjoint, we
know that $\Pi _0$ is orthogonal, $\Pi _0=\Pi _0^*$.

\par The functions $\Pi _0\widetilde{e}_j$, $j=\pm 1$ form a basis in
$\CE_0(h)$ and we have
\begin{equation}\label{pmr.6}
\Pi _0\widetilde{e}_j(x)-\widetilde{e}_j(x)=\widehat{{\mathcal O}}(e^{-\frac{1}{h}(d(U_{-j}^\delta ,x)+S_0-2\delta )}).
\end{equation}
From (\ref{pmr.4}) we see that $\Pi _0\widetilde{e}_j$ form an almost
orthonormal basis in $\CE_0(h)$ (see \cite{DiSj99} for more
details) and this basis can be orthonomalized by using the square root
of the Gram matrix (which is very close to the idenity) in order to
produce an orthonormal basis $e_1,e_{-1}$ such that
\begin{equation}\label{pmr.7}
e_j-\widetilde{e}_j=\widetilde{\CO}(e^{-\frac{1}{h}(S_0+d(U_{-j},x))})
\end{equation}
where we use the notation $\widetilde{\CO}(e^{f/h})$ for ${\mathcal O}(e^{(f-\varepsilon (\delta ))/h})$ (or $\widehat{\CO}
(e^{(f-\varepsilon (\delta ))/h})$ depending on the context) for every
fixed $\delta >0$, where $\varepsilon (\delta )\to 0$ when $\delta \to
0$. The matrix of ${{P_0}_\vert}_{\CE_0(h)}$ with respect to this
basis is
\begin{equation}\label{pmr.8}
\begin{pmatrix}\mu (h) &t(h)\\ \overline{t(h)} &\mu (h)\end{pmatrix},
\end{equation}
where 
\begin{equation}\label{pmr.9}
\mu (h)=\widetilde{\mu }(h)+\widetilde\CO(e^{-2S_0/h})
\end{equation}
is real and the tunneling coefficient fulfills
\begin{equation}\label{pmr.10}
t(h)=\widehat{\CO}(e^{-S_0/h}).
\end{equation}
See Theorem 6.10 in \cite{DiSj99}. 

\par
In many situation we have a matching lower bound on $|t(h)|$:
\begin{equation}\label{pmr.11}
1/|t(h)|=\widehat{\CO}(e^{S_0/h}).
\end{equation}
The two eigenvalues of $P_{0}(h)$ in the interval (\ref{pmr.5}) are the
ones of the matrix (\ref{pmr.8}):
\begin{equation}\label{pmr.12}
\mu _{\pm 1}(h)=\mu (h)\pm |t(h)|.
\end{equation}

We now turn to the perturbed operator $P_\varepsilon$,
 where $W\in \CC^\infty (M;{\R})\cap L^\infty (M)$ and we
assume for simplicity, that $\|W\|_{L^\infty }\le 1$. As for $\varepsilon
$, we require that
\begin{equation}\label{pmr.13}
|\varepsilon |\ll h^{N_0}.
\end{equation}
We know that the spectrum of $P_\varepsilon $ is discrete in some fixed
($h$-independent) neighborhood of $0$ when $h$ and $|\varepsilon
|$ are small enough. From the assumption (\ref{pmr.13}),
it follows that $P_\varepsilon $ has precisely two eigenvalues, counted
with their (algebraic) multiplicity, in the disc $D(\widetilde{\mu
},h^{N_0}/(2C))$ and these eigenvalues belong to the smaller disc 
$D(\mu (h),|t(h)|+\varepsilon )$. Let $\CE_\varepsilon (h)$ be the
corresponding 2-dimensional spectral subspace and let $\Pi _\varepsilon
(h):L^2(M)\to \CE_\varepsilon (h)$ be the spectral projection, where
we recall the Riesz formula
\begin{equation}\label{pmr.14}
\Pi _\varepsilon =\frac{1}{2\pi i}\int_\gamma (z-P_\varepsilon )^{-1}dz,\
\gamma =\partial D(\widetilde{\mu },\frac{h^{N_0}}{2C}).
\end{equation}
Here $D(z_0,r)$ denotes the open disc in ${\C}$ of center $z_0$
and radius $r$.
Using the Riesz formula (cf. \cite[p.62]{DiSj99}) we obtain
\begin{equation}\label{pmr.15}
\|\Pi _\varepsilon -\Pi _0\|=\CO(\varepsilon h^{-N_0})\ll 1.
\end{equation}
Thus, introducing 
\begin{equation}\label{pmr.16} e_j^\varepsilon=\Pi _\varepsilon e_j,\end{equation}
we see that $e_1^\varepsilon $, $e_{-1}^\varepsilon $ form a basis for
$\CE_\varepsilon (h)$ which is close to be orthonormal.
Differentiating in (\ref{pmr.14}), we see that
\begin{equation}\label{pmr.17}
\partial _\varepsilon \Pi _\varepsilon =\CO(h^{-N}),
\end{equation}
which also implies (\ref{pmr.15}).

\par As we have seen in Section \ref{la}, LA estimates work
also for $P_\varepsilon $ and we have
\begin{equation}\label{pmr.18}
  e_j^\varepsilon ,\, \partial _\varepsilon e_j^\varepsilon =\widetilde{\CO}(e^{-d(U_j,x)/h}). 
\end{equation} 
In fact, we know as in the self-adjoint case (\cite{DiSj99}) that $\Pi
_\varepsilon ,\, \partial _\varepsilon \Pi _\varepsilon =\widehat{{\mathcal O}}(e^{-d(x,y)/h})$ and $e_j=\CO(e^{-d(U_j,x)/h})$. The functions
$e_j^\varepsilon $, $j=\pm 1$, form an orthonormal basis for ${\CE}_\varepsilon (h)$ when $\varepsilon =0$ but not necessarily when
$\varepsilon \ne 0$. Recalling that $P_\varepsilon ^*=P_{-\varepsilon }$, we
let $f_1^\varepsilon ,f_{-1}^\varepsilon \in \CE_{-\varepsilon } (h)$ be the
dual basis to $e_1^\varepsilon ,\, \varepsilon _{-1}^\varepsilon \in {\mathcal
  E}_\varepsilon (h)$:
\begin{equation}\label{pmr.19}
(f_j^\varepsilon |e_k^\varepsilon )=\delta _{j,k},\ j,k\in \{-1,1 \}.
\end{equation}
\begin{prop}\label{pmr.20}
We have
\begin{equation}\label{pmr.21}
f_k^\varepsilon ,\, \partial _\varepsilon f_k^\varepsilon =\widetilde{{\mathcal O}}(e^{-d(U_k,x)/h}),\ k=\pm 1.
\end{equation}
\end{prop}
\begin{proof}
Let $b_{j,k}=(e_j^{-\varepsilon }|e_k^\varepsilon )$, so that in the space
of $2\times 2$-matrices,
\begin{equation}\label{pmr.22}
(b_{j,k})=1+\CO(\varepsilon h^{-N_0})
\end{equation}
by (\ref{pmr.15}). By (\ref{pmr.18}) we have
\begin{equation}\label{pmr.23}
b_{j,k},\, \partial _{\varepsilon }b_{j,k}=\widetilde{\CO } (e^{-S_0/h}),\hbox{
  when }j\ne k.
\end{equation}
Write
$$
f_j^\varepsilon =\sum_\nu c_{j,\nu }e_{\nu }^{-\varepsilon }.
$$
Then (\ref{pmr.19}) reads
$$
\sum_\nu  c_{j,\nu }(e_\nu ^{-\varepsilon }|e_k^\varepsilon )=\delta _{j,k},
$$
i.e. 
$$
\sum_\nu  c_{j,\nu }b_{\nu ,k}=\delta _{j,k},
$$
so
\begin{equation}\label{pmr.24}
(c_{j,k})=(b_{j,k})^{-1}=\begin{pmatrix}1/b_{1,1} &0\\ 0
  &1/b_{-1,-1}\end{pmatrix}+\widetilde{\CO}(e^{-S_0/h}),\
b_{j,j}=1+\CO(\varepsilon h^{-N_0}),
\end{equation}
where the last equality follows from (\ref{pmr.22}). We
therefore get the estimate for $f_k^\varepsilon $ in (\ref{pmr.21}).

\par In order to get the estimate for $\partial _\varepsilon f_k^\varepsilon
$ in (\ref{pmr.21}), we first observe that 
\begin{equation}\label{dw.38}
\partial _\varepsilon b_{j,j}=\CO(h^{-N_0}),\ \partial _\varepsilon
b_{j,k}=\widetilde{\CO}(e^{-S_0/h}),\hbox{ when }j\ne k.
\end{equation}
Combining this with the standard formula
$$
\partial _\varepsilon (c_{j,k})=-(c_{j,k})\circ \partial _{\varepsilon
}(b_{j,k})\circ (c_{j,k}),
$$
(\ref{pmr.22}) and (\ref{pmr.24}), we see that $c_{j,k}$ also satisfy
(\ref{dw.38}):
\begin{equation}\label{pmr.26}
\partial _\varepsilon c_{j,j}=\CO(h^{-N_0}),\ \partial _\varepsilon
c_{j,k}=\widetilde{\CO}(e^{-S_0/h}),\hbox{ when }j\ne
k. \end{equation}
Now, 
$$
\partial _\varepsilon f_k^\varepsilon =\sum_\nu (\partial _\varepsilon c_{k,\nu
})e_\nu ^{-\varepsilon }+\sum_\nu c_{k,\nu } (\partial _\varepsilon e_\nu
^{-\varepsilon })
$$
and the estimate for $\partial _\varepsilon f_k^\varepsilon $ in
(\ref{pmr.21}) follows from (\ref{pmr.22}), (\ref{pmr.24}),
(\ref{pmr.18}) with $\varepsilon $ replaced by $-\varepsilon $ in the last
relation. 
\end{proof}

\par Let $M_\varepsilon =(m_{j,k}^\varepsilon )$ denote the matrix of
$P_\varepsilon =\CE_\varepsilon (h)\to \CE_\varepsilon (h) $ with
respect to the basis $e_1^\varepsilon ,\, e_{-1}^\varepsilon $. Then
\begin{equation}\label{pmr.27}
m_{j,k}^\varepsilon = (P_\varepsilon e_k^\varepsilon |f_j^\varepsilon
)=(e_k^\varepsilon |P_{-\varepsilon }f_j^\varepsilon ).
\end{equation}
Note that $f_j^0=e_j^0$ since $e_1^0,\, e_{-1}^0$ is an orthonormal
basis, and that $M_0$ is the matrix in (\ref{pmr.8}).

\par Naturally, the ${\CP\CT}$-symmetry of $P_\varepsilon $ induces a
corresponding symmetry for $M_\varepsilon $ that we shall make
explicit. By construction, we have ${\CP\CT}e_j^\varepsilon
=e_{-j}^\varepsilon $. Also notice that 
$$
({\CP\CT}u|{\CP\CT}v)=\overline{(u|v)}=(v|u),\ u,v\in L^2(M).
$$
From  (\ref{pmr.19}), we get
$$
({\CP\CT}f_j^\varepsilon |{\CP\CT}e_k^\varepsilon )=\delta _{j,k},
$$
i.e.
$$
({\CP\CT}f_j^\varepsilon |e_{-k}^\varepsilon )=\delta _{j,k}=\delta _{-j,-k}.
$$
Comparing with (\ref{pmr.19}) (and recalling that $\CE_\varepsilon $
and $\CE_{-\varepsilon} $ are invariant under the action of ${\CP\CT}$) we conclude that 
\begin{equation}\label{pmr.28}
{\CP\CT}f_j^\varepsilon =f_{-j}^\varepsilon .
\end{equation}
We have,
\begin{equation}\label{pmr.29}
\begin{split}
m_{j,k}^\varepsilon &=(P_\varepsilon e_k^\varepsilon |f_j^\varepsilon )=(P_\varepsilon
{\CP\CT}e_{-k}^\varepsilon |{\CP\CT}f_{-j}^\varepsilon )\\
&=({\CP\CT}P_\varepsilon e_{-k}^\varepsilon |{\CP\CT}f_{-j}^\varepsilon )
=\overline{(P_\varepsilon e_{-k}^\varepsilon |f_{-j}^\varepsilon
  )}=\overline{m_{-j,-k}^\varepsilon },
\end{split}
\end{equation}
which means that the general form of $M_\varepsilon $ is
\begin{equation}\label{pmr.30}
M_\varepsilon =\begin{pmatrix}a(\varepsilon ) &b(\varepsilon )\\
  \overline{b}(\varepsilon ) &\overline{a}(\varepsilon )\end{pmatrix}.
\end{equation}
This can also be expressed as a ${\CP\CT}$-symmetry property of
$M_\varepsilon $ as a linear map: ${\C}^2\to {\C}^2$: Define $\pi
,\tau :{\C}^2\to {\C}^2$ by 
\begin{equation}\label{pmr.31}
\pi \begin{pmatrix}x_1\\x_2\end{pmatrix}=\begin{pmatrix}x_2\\x_1\end{pmatrix},\quad
\tau \begin{pmatrix}x_1\\
  x_2\end{pmatrix}=\begin{pmatrix}\overline{x}_1\\ \overline{x}_2\end{pmatrix}.
\end{equation}
Then (\ref{pmr.29}) is equivalent to the property,
\begin{equation}\label{pmr.32}
\pi \tau M_\varepsilon =M_\varepsilon \pi \tau .
\end{equation}
Since this formulation will not be needed below, we leave out the
simple and straight forward proof.

\par We now study $\partial _\varepsilon m_{j,k}^\varepsilon $. First, if
$j\ne k$, we have
\begin{equation}\label{pmr.33}
\partial _\varepsilon m_{j,k}^\varepsilon =i(We_k^\varepsilon |f_j^\varepsilon
)+(P_\varepsilon \partial _\varepsilon e_k^\varepsilon |f_j^\varepsilon
)+(P_\varepsilon e_k^\varepsilon |\partial _\varepsilon
f_j^\varepsilon )=\widetilde{\CO } (e^{-S_0/h}).
\end{equation}

\par For $j=k$, we start with
\begin{equation}\label{pmr.34}
\partial _\varepsilon m_{j,j}^\varepsilon =i(We_j^\varepsilon |f_j^\varepsilon
)+(P_\varepsilon \partial _\varepsilon e_j^\varepsilon |f_j^\varepsilon
)+(P_\varepsilon e_j^\varepsilon |\partial _\varepsilon f_j^\varepsilon ).
\end{equation}
Here we use that $e_j^\varepsilon =e_j^0+\CO(\varepsilon h^{-N_0})$,
$f_j^\varepsilon =f_j^0+\CO(\varepsilon h^{-N_0})$ in $L^2$, to see
that 
\begin{equation}\label{pmr.35}
(We_j^\varepsilon |f_j^\varepsilon )=(We_j^0|e_j^0)+\CO(\varepsilon
h^{-N_0})=\int W(x)|e_j^0(x)|^2 dx +\CO(\varepsilon h^{-N_0}).
\end{equation}
In order to treat the other two terms in (\ref{pmr.34}), we recall that
by definition of $m_{j,k}^\varepsilon $, we have
\begin{equation}\label{pmr.36}
P_\varepsilon e_j^\varepsilon =\sum_\nu m_{\nu ,j}^\varepsilon e_\nu ^\varepsilon .
\end{equation}
We need a similar formula for $P_\varepsilon ^*f_j^\varepsilon $, so we take
the $L^2$ inner product of (\ref{pmr.36}) with $f_k^\varepsilon $ and get
$$
(e_j^\varepsilon |P_\varepsilon ^*f_k^\varepsilon )=\sum_\nu m_{\nu ,j}^\varepsilon
\underbrace{(e_\nu ^\varepsilon |f_k^\varepsilon )}_{\delta
  _{\nu ,k}}=m_{k,j}^\varepsilon .
$$
Exchange $j,k$ and take the complex conjugates:
$$
(P_\varepsilon ^*f_j^\varepsilon |e_k^\varepsilon )=\overline{m}_{j,k}^\varepsilon ,
$$
to conclude that
\begin{equation}\label{pmr.37}
P_\varepsilon ^*f_j^\varepsilon =\sum_\nu \overline{m}_{j,\nu }^\varepsilon
f_\nu ^\varepsilon .
\end{equation}
Using (\ref{pmr.36}), (\ref{pmr.37}), we get
\[\begin{split}
(P_\varepsilon \partial _\varepsilon e_j^\varepsilon |f_j^\varepsilon
)+(P_\varepsilon e_j^\varepsilon |\partial _\varepsilon f_j^\varepsilon
)&=(\partial _\varepsilon e_j^\varepsilon |P_\varepsilon ^* f_j^\varepsilon
)+(P_\varepsilon e_j^\varepsilon |\partial _\varepsilon f_j^\varepsilon )\\
&=\sum_\nu \left( m_{j,\nu }^\varepsilon (\partial _\varepsilon e_j^\varepsilon
|f_\nu ^\varepsilon ) +m_{\nu ,j}^\varepsilon (e_\nu ^\varepsilon |\partial
_\varepsilon f_j^\varepsilon )\right)\\
&=m_{j,j}^\varepsilon \underbrace{\left( (\partial _\varepsilon e_j^\varepsilon
  |f_j^\varepsilon )+(e_j^\varepsilon |\partial _\varepsilon f_j^\varepsilon )
\right)}_{\partial _\varepsilon (e_j^\varepsilon |f_j^\varepsilon )=\partial
_\varepsilon (1)=0}
\\ &\hskip 2 truecm +\underbrace{m_{j,-j}^\varepsilon }_{\widetilde{\CO}(e^{-S_0/h})}
\underbrace{(\partial _\varepsilon e_j^\varepsilon |f_{-j}^\varepsilon )} _{\widetilde{\CO}(e^{-S_0/h})}+\underbrace{m_{-j,j}^\varepsilon }_{\widetilde{\CO}(e^{-S_0/h})}
\underbrace{(e_{-j}^\varepsilon |\partial _\varepsilon f_{j}^\varepsilon )}
_{\widetilde{\CO}(e^{-S_0/h})}\\
&=\widetilde{\CO}(e^{-2S_0/h}).
\end{split}\]
Combining this with (\ref{pmr.34}), (\ref{pmr.35}), we obtain
\begin{equation}\label{pmr.38}
\partial _\varepsilon m_{j,j}^\varepsilon =i\int W(x)|e_j^0(x)|^2dx+{\mathcal O}(\varepsilon h^{-N_0})+\widetilde{\CO}(e^{-2S_0/h})
\end{equation}
and by integration in $\varepsilon $ (cf. (\ref{pmr.30}), (\ref{pmr.8})), 
\begin{equation}\label{pmr.39}
a(\varepsilon )=\mu (h)+i\varepsilon \int W(x)|e_j^\varepsilon
(x)|^2dx+{\mathcal O}(\varepsilon ^2h^{-N_0})+\varepsilon
\widetilde{\CO}(e^{-2S_0/h}).
\end{equation}

\par By (\ref{pmr.33}), we have
\begin{equation}\label{pmr.42}
\partial _\varepsilon b,\, \partial _\varepsilon |b|=\widetilde{\CO } (e^{-S_0/h}),
\end{equation}
which implies that
\begin{equation}\label{pmr.40}
b(\varepsilon )=t(h)+\varepsilon \widetilde{\CO}(e^{-S_0/h}).
\end{equation}

\par The eigenvalues of ${{P_\varepsilon }_\vert}_{\CE_\varepsilon
  (h)}$ are equal to the ones of $M_\varepsilon $ (cf. (\ref{pmr.30})):
\begin{equation}\label{pmr.41}
\lambda_{\pm} =\Re a\pm\sqrt{|b|^2-(\Im a)^2}.
\end{equation}

Assume now that
\begin{equation}\label{pmr.43}
W>0\hbox{ on }U_1
\end{equation}
and hence also on a fixed neighborhood of that set. Since $e_1^0$ is
exponentially concentrated to a neighborhood of $U_1$, we conclude
that 
\begin{equation}\label{pmr.44}
\int W(x)|e_1^0(x)|^2 dx \asymp 1,
\end{equation}
and (\ref{pmr.38}) shows that
\begin{equation}\label{pmr.45}
\partial _\varepsilon \Im a =\int W|e_1^0|^2 dx+\CO(\varepsilon
h^{-N_0})+\widetilde{\CO}(e^{-2S_0/h})\asymp 1. 
\end{equation}

We can now discuss when the two eigenvalues (cf. (\ref{pmr.41})) are
real or complex. Since we are dealing with a ${\CP\CT}$ symmetric
operator, we know that the eigenvalues are either real or form complex
conjugate pairs. This means that $P_{-\varepsilon }=P_\varepsilon^* $ and
$P_\varepsilon $ have the same spectrum. Consequently, we can restrict
the attention to the region $0\le \varepsilon \ll h^{N_0}$. The reality
or not of our two eigenvalues is determined by the sign of 
\begin{equation}\label{pmr.46}
|b|-(\Im a)^2=(|b|+\Im a)(|b|-\Im a).
\end{equation}
Recall that $\Im a$ vanishes when $\varepsilon =0$ and is a strictly
increasing function of $\varepsilon $ whose derivative is $\asymp 1$,
while $b(\varepsilon )$ and its derivative with respect to $\varepsilon $ are
exponentially small. Thus, if we first consider the case when $t(h)=0$,
we see that both factors in (\ref{pmr.46}) vanish for $\varepsilon =0$
(corresponding to a double real eigenvalue of $P_0$) and for $\varepsilon
>0$ the first factor is positive while the second one is negative, so
the two eigenvalues in (\ref{pmr.41}) are non-real and complex
conjugate for $\varepsilon >0$.

Let now $t(h)\ne 0$ (but still exponentially small as we recalled in
(\ref{pmr.10})). Then the first factor in (\ref{pmr.46}) is strictly
positive for $0\le \varepsilon \ll h^{N_0}$. Denote the second factor by 
$f(\varepsilon )=|b|-\Im a$. Then $f(0)=|t(h)|>0$ and
\begin{equation}\label{pmr.47}
f'(\varepsilon )=-\int W(x) |e_j^0|^2dx+\CO(\varepsilon
h^{-N_0})+\widetilde{\CO}(e^{-S_0/h})\asymp -1.
\end{equation}
Hence there exists a point $\varepsilon _+(h)>0$ such that $f(\varepsilon
)>0$ for $0\le \varepsilon <\varepsilon _+$, $f(\varepsilon _+)=0$, $f(\varepsilon
)<0$ for $e_+<\varepsilon \ll h^{N_0}$. In the first region we have two
real and distinct eigenvalues, at the point $\varepsilon _+$ we have a
real  double eigenvalue, while in the last region we have a pair of complex
conjugate non-real eigenvalues.

In view of (\ref{pmr.10}) and (\ref{pmr.47}) we know that $\varepsilon
_+(h)=\widehat{\CO}(e^{-S_0/h})$ and if we restrict the attention
to the exponentially small interval $[0,2\varepsilon _+]$ we can sharpen
(\ref{pmr.47}) to 
$$
f'(\varepsilon )=-\int W(x)|e_j^0(x)|^2 dx+\widetilde{\CO}(e^{-S_0/h}),
$$
which implies that
\begin{equation}\label{pmr.48}
\varepsilon _+=(1+\widetilde{\CO}(e^{-S_0/h}))\frac{|t(h)|}{\int W(x)|e_1^0(x)|^2dx},
\end{equation}
and this finishes the proof of Theorem \ref{int.18}.

\appendix
\section{The spectrum of $P_{\varepsilon}$}
\label{app}

We recall from the Introduction that $P_0$ denotes the Friedrichs extension of the differential operator $-h^2\Delta+V_{0}$ from $\CC_0^\infty (M)$, $M=\R^n$ or a Riemannian compact manifold. In the first case
$$
\alpha=\liminf_{x\to \infty} V_{0}(x),
$$
and $\alpha=+\infty$ in the latter case.
We recall that the domain $\CD(P_0)$ of $P_{0}$ contains the form domain
$$
\{ u\in L^2(M);\, \int |\nabla u|^2dx+\int (V_0)_+(x)\vert u\vert^2 dx<+\infty  \},
$$
where $(V_0)_+(x)=\max (V_0(x),0)$.

\begin{prop}\label{app.1}
\sl
The spectrum of $P_\varepsilon $ in the left half-plane $\Re z<\alpha $
is discrete.
\end{prop}

\begin{proof}
When $M$ is compact this follows quite easily from the ellipticity of
$P_\varepsilon $ and the fact that there are always points with $\Re z\ll
0$ that do not belong to the spectrum. 

Thus, we consider the case when $M={\R}^n$. Let $\beta <\alpha $ be
arbitrarily close to $\alpha $ and put $V_{0,\beta }(x)=\max
(V_0(x),\beta )$ so that $V_{0,\beta }$ is equal to $V_0$ near
infinity or equivalently so that $\mathrm{supp\,}(V_{0,\beta }-V_0)$
is compact. Put $P_{\varepsilon ,\beta }=-h^2\Delta +V_{0,\beta
}(x)+i\varepsilon W(x)$. 

\par Let us first notice that $P_{\varepsilon ,\beta }-z:\CD(P_0)\to
L^2$ is bijective with bounded inverse when $\Re z<\beta $. Indeed,
the injectivity follows from the estimate
$$
\Re ((P_{\varepsilon ,\beta }-z)u|u)\ge ((V_{0,\beta }-\Re z)u|u)\ge
(\beta -\Re z)\|u\|^2,\ u\in \CD(P_0).
$$
Notice also from this that $P_{\varepsilon ,\beta }-z$ has a bounded left
inverse $R_{\varepsilon ,\beta }(z)$ of norm $\le (\beta -\Re z)^{-1}$ in
$\CL(L^2,L^2)$. When $\varepsilon =0$, $P_{0,\beta }$ is
self-adjoint and $P_{0,\beta }-z$ is bijective, so the left inverse
is a bilateral inverse. By a simple deformation argument in $\varepsilon
$ we get the claimed bijectivity for all $\varepsilon $. 

\par Still for $\Re z<\beta $ we write
\[
P_\varepsilon -z=P_{\varepsilon ,\beta }-z+(V_0-V_{0,\beta })=\begin{cases}
  (P_{\varepsilon ,\beta }-z)(1+(P_{\varepsilon ,\beta
  }-z)^{-1}(V_0-V_{0,\beta }))\\
  \hbox{and also}\\
  (1+(V_0-V_{0,\beta }) (P_{\varepsilon ,\beta }-z)^{-1}) (P_{\varepsilon
    ,\beta }-z).
\end{cases}
\]
Here $(V_0-V_{0,\beta }):\CD(P_0)\to L^2$ is compact, since
$V_0-V_{0,\beta }\in L^\infty _\mathrm{comp}$, so
\[\begin{split}&(P_{\varepsilon ,\beta }-z)^{-1}(V_0-V_{0,\beta }):\, \CD(P_0)\to
\CD(P_0),\\ &(V_0-V_{0,\beta }) (P_{\varepsilon ,\beta
}-z)^{-1}:\, L^2\to
L^2 \end{split}\]
are compact. The operator norms of these operators are
$\CO((\beta -\Re z)^{-1})$. Thus 
$$1+(P_{\varepsilon ,\beta
  }-z)^{-1}(V_0-V_{0,\beta }):\, \CD(P_0)\to \CD(P_0)$$
and 
$$
1+(V_0-V_{0,\beta })(P_{\varepsilon ,\beta }-z)^{-1}:\, L^2\to L^2
$$
are holomorphic families of Fredholm operators of index 0, bijective
when $\Re z\ll 0$. From these observations we get the
proposition in a fairly standard way.
\end{proof}


\end{document}